\documentclass[12pt]{amsart}
\usepackage{amsmath}
\usepackage{amsfonts,amssymb,amsthm, txfonts, pxfonts}
\usepackage[OT2,OT1]{fontenc}
\newcommand\cyr{%
\renewcommand\rmdefault{wncyr}%
\renewcommand\sfdefault{wncyss}%
\renewcommand\encodingdefault{OT2}%
\normalfont
\selectfont}
\DeclareTextFontCommand{\textcyr}{\cyr}
\def\struckint{\mathop{%
\def\mathpalette##1##2{\mathchoice{##1\displaystyle##2}%
  {##1\textstyle##2}{##1\scriptstyle##2}{##1\scriptscriptstyle##2}}%
\mathpalette
{\vbox\bgroup\baselineskip0pt\lineskiplimit-1000pt\lineskip-1000pt
\halign\bgroup\hfill$}
{##$\hfill\cr{\intop}\cr\diagup\cr\egroup\egroup}%
}\limits}

\newcommand{\ahn}{{\mathbb{H}^n}}

\newcommand{\hypvol}[1]{V_\mathbb{H}^#1}
\newcommand{\hyplam}[1]{\Lambda_\mathbb{H}^#1}
\newcommand{\sphvol}[1]{V_\mathbb{S}^#1}

\newcommand{\eucvol}[1]{V_\mathbb{E}^#1}

\newcommand{\genvol}[1]{V_\mathbb{X}^#1}
\newcommand{\genlam}[1]{\Lambda_\mathbb{X}^#1}

\newtheorem{theorem}{Theorem}[section]
\newtheorem{lemma}[theorem]{Lemma}
\newtheorem{corollary}[theorem]{Corollary}

\newtheorem{definition}[theorem]{Definition}

\newtheorem{theorem-definition}[theorem]{Theorem-Definition}

\theoremstyle{remark}

\newtheorem{remark}[theorem]{Remark}

\newtheorem{example}[theorem]{Example}

\newtheorem{question}[theorem]{Question}

\newcommand{\reals}{\mathbb{R}}
\newcommand{\sphere}{\mathbb{S}}

\DeclareMathOperator{\asin}{asin}
\DeclareMathOperator{\acosh}{arccosh}

\DeclareMathOperator{\lobfn}{\textcyr{L}}

\begin{document}


\title{Asymptotics of convex sets in $\mathbb{E}^n$ and $\ahn.$}


\author{Igor Rivin}

\address{Department of Mathematics, Temple University, Philadelphia}

\email{rivin@math.temple.edu}

\thanks{The author would like to thank  Stanford
  University for their hospitality during the preparation of this work. He would also like to thank Itai Benjamini for the question that led to this line of inquiry, Jean-Marc Schlenker for enlightening correspondence, and to Peter Storm for helpful conversations.}

\date{\today}


\begin{abstract}
We study convex sets  $C$ of finite (but non-zero) volume in $\ahn$ and $\mathbb{E}^n.$ We show that the intersection $C_\infty$ of any such set with the ideal boundary of $\ahn$ has Minkowski (and thus Hausdorff) dimension of at most $(n-1)/2.$ and this bound is sharp, at least in some dimensions $n.$ We also show a sharp bound when $C_\infty$ is a smooth submanifold of $\partial_\infty \ahn.$ In the hyperbolic case, we show that for any   $k\leq (n-1)/2$ there is a bounded section $S$ of $C$ through any prescribed $p,$ and we show an upper bound on the radius of the ball centered at $p$ containing such a section. We show similar bounds for sections through the origin of a convex body in $\mathbb{E}^n,$ and give asymptotic estimates as $1\ll k \ll n.$ .\end{abstract} 

\maketitle

\section*{Introduction}
The work in this note was motivated by a question of Itai Benjamini and Nir Avni on whether there is any version of A. Dvoretsky's Theorem valid in high-dimensional hyperbolic spaces $\ahn.$ It quickly became apparent that in order to have any hope of answering this question one must have a good understanding of the geometry of convex sets in $\ahn$ at (and near) the ideal boundary, and this is the subject of this work. The most basic question of this type is to understand the geometry of the ``ideal part'' $C_\infty$ of a convex set $C$ with nonempty interior and finite volume (a simpler way of putting it is requiring $0 < V(C) < \infty.$) The first most basic question is: what is the \emph{dimension} of $C_\infty?$ There are, of course, many definitions of dimension, but the most natural one for our purposes turns out to be the (upper) Minkowski dimension $\overline{\dim}_M.$ Using a simple geometric idea we show that
\[\overline{\dim}_M(C_\infty) \leq \dfrac{n-1}{2},\]  Since the Minkowski dimensions are both upper bounds on the Hausdorff dimension, we have the same bound on the Hausdorff dimension.  We show further that for $C_\infty$ smooth, the volume of the convex hull of $C_\infty$ is finite whenever the (topological) dimension of $C_\infty$ is not greater than $\lfloor n/2\rfloor - 1,$ and that bound is sharp.

In dimension $3,$ we show that there are sets $C_\infty,$ of arbitrary Hausdorff dimension smaller than $1,$ such that the volume of the convex hull of $C_\infty$ is finite. We do not know whether there are sets of Hausdorff dimension equal to $1$ with that property.

The next question is whether there is always a $k$-dimensional plane through any fixed point $p$ of $C$ of bounded diameter. The dimension estimate above essentially shows that the answer is affirmative (whenever $k$ does not exceed the critical dimension $(n-1)/2$, but with some extra work we can show more precise estimates on exactly how small we can get the diameter in terms of $n, k, V(C)$ and the ``thickness'' of $C$ (that is, the radius of the largest ball centered on $p$ and contained in $C.$) The nature of the argument is such that we can, essentially without change, obtain estimates of the sort "intersections of at least 30\% of all planes through $p$ are contained in $B(p, r).$ 

The basic idea is simple: if we let $\Omega_r(C)$ be the set of directions in which rays of length $r$ emanating from $p$ are contained in $C,$ then in order to estimate the measure of the set of planes which intersect $\Omega_r(C)$ we produce a bound on the measure of the $\epsilon$-neighborhood of $\Omega_r(C)$ (which is always a Borel set, unlike $\Omega_r(C)$ itself). To produce such a bound we use a couple of simple geometric ideas, the first (trivial) one giving a bound on $\Omega_r(C)$ as a function of $r,$ and secondly, using the Double Cone Lemma (in Section \ref{doubleconesec}) we show that a certain $\epsilon$-neighborhood of $\Omega_r(C)$ is contained in an $\Omega_s(C),$ for some $s(r, \epsilon).$

The outline of the paper is as follows:

In Section \ref{minky} we recall the basic definitions of Minkowski measure and content.

In Section \ref{ballvol} we recall some basic formulas and estimates on the volumes of balls in $\ahn.$

In Section \ref{kleinmod} we recall some of the properties of the Klein model of $\ahn.$

In Section \ref{doubleconesec} we describe our basic geometric tool -- the ``Double Cone Lemma.''

In Section \ref{limapps} we prove the basic estimates on the limit sets of finite volume convex sets in $\ahn.$ Our main results are Theorem \ref{hdimest}, which states that the upper Minkowski dimension (hence the Hausdorff dimension) of the limit set of a convex set of finite volume in $\ahn$ is bounded above by $(n-1)/2,$ and Theorem \ref{smoothvol}, which observes that the volume of a convex hull of a smooth subset $S$ of the ideal boundary of $\ahn$ is finite if and only if the (topological) dimension of $S$ is no greater than $\lfloor n/2-1\rfloor.$

In Section \ref{explower} we construct a family of sets in $\partial \mathbb{H}^3$ of Hausdorff dimension tending to $1,$ such that the volume of the hyperbolic convex hull of each of the set is finite, showing that the result of Section \ref{limapps} are sharp (at least in dimension 3).

In Section \ref{centralsec} we study the sizes of the intersections of a non-degenerate convex set $C\in \ahn$ with planes through a fixed point $p.$ The main result is Theorem \ref{hypradversion}, which is too cumbersome to state here, but implies (for large $n$ and $C$ of large volume $V(C)$) that one can find such a section of dimension $1 \ll k \ll n$ contained in a ball of radius about $\frac12 \log n$ bigger than the radius of a round ball of volume $V(C)$ in $\ahn.$ 

In Section \ref{euclid} we apply our methods to similar questions in Euclidean space, where, not surprisingly, the estimates come out quite differently. The main result is Theorem \ref{mainthmeuc}, which implies that for a $1\ll k \ll n,$ there is a section of $C$ contained in a ball of radius about $\sqrt{n/2\pi e}$ bigger than the radius of a round ball in $\mathbb{E}^n$ of volume $V(C).$

In Section \ref{theests} we prove the basic technical estimates we need.
\subsection{Notation}
\label{notation}
We shall denote the volume of a ball of radius $r$ in $\mathbb{E}^n,\mathbb{S}^n, \mathbb{H}^n$ by $\eucvol{n}(r),\sphvol{n}(r), \hypvol{n(r)},$ respectively. In addition we will use the notation $\genlam{n}(V)$ for the \emph{inverse function} of $\genvol{n},$ for $X=\mathbb{E}, \mathbb{H}, \mathbb{S}$ -- that is, $\genlam{n}(\genvol{n}(r)) = r.$ We will also use the standard notation $\kappa_n = \eucvol{n}(1),$ and also $\omega_n$ for the area of the sphere of unit radius in $\mathbb{E}^n.$  As in the previous sentence, we will use $X$ when the statement does not depend on which of the three ambient spaces we are talking about.

We will frequently use the following function:
\begin{definition}
\label{alphadef}
Let $r_1 > r_2 >r_0.$
Then, we define
\[
\alpha_{r_0}(r_1, r_2) = \asin\left(r_0 \dfrac{\sqrt{r_1^2 - r_0^2} - \sqrt{r_2^2 - r_0^2}}{r_2}\right).
\]
\end{definition}

We will denote the $\epsilon$ neighborhood of a subset $S$ of $\mathbb{S}^k$ by $S_\epsilon.$
In some places below we use the notation $\mu(S)$ for subsets of $\mathbb{S}^k$ not assumed Lebesgue measurable. In such cases $\mu$ stands for the \emph{lower Minkowski content} of $S,$ namely 
\[
\mu(S) = \lim \inf_{\epsilon \rightarrow 0} \lambda(S_\epsilon),
\]
where $\lambda$ is Lebesgue measure. We will also use the notation $\nu(S)$ for the \emph{normalized probability} measure of $S$ (in other words, $\nu(S) = \mu(S)/\omega_{k+1},$ so that $\nu(\mathbb{S}^k) = 1.$  For discussion of Minkowski content (and all other measure-theoretic concepts), the reader is referred to P.~Mattila's book \cite{mattilabook}.

\section{Minkowsky measure and dimension}
\label{minky}
This section is shamelessly stolen from P.~Mattila's book \cite{mattilabook}; we include it here in an attempt to keep this paper self-contained.

The setup is as follows: Let $A$ be a non-empty bounded subset of $\reals^n$ or $\mathbb{S}^n.$ Denoting the $n$-dimensional Lebesgue measure by $\lambda,$ as before, we define the upper $s$-dimensional  Minkowski content of $A$ by 
\[
\mathcal{M}^{*s} = \limsup_{\epsilon \rightarrow 0}(2\epsilon)^{s-n} \lambda(A_\epsilon),
\]
and the lower $s$-dimensional Minkowski content by
\[
\mathcal{M}_*^s = \liminf_{\epsilon \rightarrow 0}(2\epsilon)^{s-n} \lambda(A_\epsilon).
\]
Using these, we can define the upper Minkowski dimension as:
\[
\overline{\dim}_MA = \inf \{s:\mathcal{M}^{*s}(A) = 0\} = \sup\{s:\mathcal{M}^{*s}(A) > 0\}.
\]
Similarly, the lower Minkowski dimension is:
\[
\underline{\dim}_MA = \inf \{s:\mathcal{M}^{*s}(A) = 0\} = \sup\{s:\mathcal{M}^{*s}(A) > 0\}.
\]

\section{Geometry of Balls and Spheres} 
\label{ballvol}
Recall that:
\begin{gather}
\label{eucvols}
\kappa_n = \dfrac{\pi^{n/2}}{\Gamma(n/2+1)},\\
\omega_n = n \kappa_n.
\end{gather}
The following is also classical:\begin{theorem}
\label{sph}
\begin{gather*}
\eucvol{n}(r) = \omega_n \int_0^r r^{n-1} d x = \kappa_n r^n,\\
\hypvol{n}(r) = \omega_n \int_0^r \sinh^{n-1} d x,\\
\sphvol{n}(r) = \omega_n \int_0^r \sin^{n-1} d x.
\end{gather*}
\end{theorem}
\begin{lemma}
\label{subsetdir}
Let $\Omega \subseteq \mathbb{S}^{n-1}\subset \mathbb{E}^n =
T_p(X),$ and let $C_\Omega(r)$ be the cone over
$\Omega,$ that is, the convex hull of $r\Omega$ and the origin under
the exponential map (in particular, if $\Omega = \mathbb{S}^{n-1},$
$C_\Omega(r)$ is just the ball of radius $r$ around $p.$
The, the volume of $C_\Omega(r)$ satisfies 
\[V(C_\Omega(r)) \geq\mu(\Omega)\genvol{n}(r) = \nu(\Omega)\omega_n V^n(r),\]
with equality if $\Omega(r)$ is Lebesgue-measurable
\end{lemma}
\begin{proof}
The statement is immediate for Lebesgue-measurable sets, and for general sets the inequality is a direct consequence of the definition of the Minkowski content $\mu.$
\end{proof}
\begin{corollary}
\label{rgenmeas}
Let $C$ be a convex body in $X,$ and $p\in C.$ Let $\Omega_R(C)$ be the set of those unit $\theta$ in the unit tangent
sphere at $p$ for which the exponential map of the segment from the
origin to $R \theta$ is contained in $K.$ Then,
\[
\mu(\Omega_R(C))\leq V(C)/\genvol{n}(r) \]
and thus
\[
\nu(\Omega_R(C))\leq V(C)/\omega_n \genvol{n}(r).
\]
\end{corollary}
\begin{proof}
The cone of radius $R$ over $\Omega_R(C)$ is contained in $C,$ so the estimate of Lemma \ref{rgenmeas} applies.
\end{proof}
\subsection{Volume asymptotics}
\begin{lemma}
\label{sphasymp}
As $r$ goes to infinity, $\hypvol{n}(r)$ is asymptotic to 
\[
\dfrac{\omega_n e^{(n-1) r}}{2^{n-1}(n-1)};
\]
for all $r > \log 2 / 2,$
\[
\dfrac{\omega_n e^{(n-1) r}}{2^{n-1}(n-1)} > \hypvol{n}(r) > \dfrac{\omega_n e^{(n-1) r}}{4^{n-1}(n-1)}. 
\]
\end{lemma}
\begin{proof}
Immediate from Theorem \ref{sph}.
\end{proof}
\begin{remark}
\label{inverserem}
Lemma \ref{sphasymp} can be thought of stating that a ball in $\ahn$ of large volume $V$ has radius
\[
r =\hyplam{n}(V) \sim  \dfrac{\log \left(\dfrac{2^{n-1}(n-1)V}{\omega_n}\right)}{n-1}.
\]
For $n \gg 1,$ Stirling's formula tells us that 
\begin{equation}
\label{stirlingball}
r \sim \log 2 /2  - \log \pi/2 -1/2+\log n/2 + \log(V)/(n-1).
\end{equation}
\end{remark}

\section{The Klein model of $\ahn$} 
\label{kleinmod}
The \emph{Klein Model}
\[
K: \ahn \rightarrow B^n(0, 1)
\]
is a representation of $\ahn$ as the interior of the unit ball in
$\mathbb{E}^n.$ It has the virtue that it is \emph{geodesic}, so that
the images of totally geodesic subspaces of $\ahn$ are intersections
of affine subspaces of $\mathbb{E}^n$ with $B^n(0, 1).$ Consequently,
the images of convex sets of $\ahn$ under $K$ are also convex. 
The hyperbolic metric can be recovered from $B^n(0, 1)$ as follows:

If $p, q \in B^n(0, 1),$ then, denoting the hyperbolic distance
between $K^{-1}(p), K^{-1}(q)$ by $d_H(p, q),$ we have the
formula\footnote{A geometric way to understand the below formula is as
  \emph{Hilbert distance} on the ball -- if the line through $p, q$
  intersects the unit sphere at $u, v,$ then $d_H(p, q) = 1/2 \log([u,
  p, q, v]),$ where $[]$ denotes the cross-ratio.}
\[
d_H(p, q) = \acosh \left(\dfrac{1- p \cdot q}{\sqrt{1-p \cdot
      p}\sqrt{1-q\cdot q}}\right).
\]
In particular, if $p_0 = K^{-1}(\mathbf{0}),$ then
\begin{equation}
\label{distK}
d_H(p_0, q) = \dfrac{1}{2} \log\left(\dfrac{1+\|q\|}{1-\|q\|}\right).
\end{equation}
Conversely, if $p, q \in \ahn$ and $K(p)=0,$ and $d(p, q) = R,$ then 
\begin{equation}
\label{distKinv}
\|K(q)\| = \tanh R.
\end{equation}
The hyperbolic metric can be expressed (see, eg, \cite{milnorvol}) as
follows in the Klein model. First, we use polar coordinates:
\[
d\mathbf{x} = d r^2 + r^2 \|d\mathbf{u}\|^2.
\]
Hyperbolic metric is then written as:
\[
\label{milnoreq}
ds^2 = \dfrac{dr^2}{(1-r^2)^2}+ \dfrac{r^2\|d\mathbf{u}\|^2}{1-r^2},
\]
showing that 
$K^{-1}$ at 
$q$ distorts distances by a factor of \[
\dfrac{1}{\sqrt{1/\|q\|^2-1}},
\]
in the spherical direction, but by a factor of \[\dfrac1{1-\|q\|^2}\] radially.

Finally, if we define $\Omega_d^K(C)$ to be the set on the visual
sphere of $0$ of the points of a convex body $C$ outside the
(Euclidean) ball of radius $d\gg 1,$ then the formula \eqref{distK}
together with Corollary \ref{rgenmeas} tell us:
\begin{lemma}
\label{Kmeas}
With definitions as above,
\[
\mu(\Omega_d^K(C)) \leq \dfrac{(n-1) 2^{n-1}V(C)}{\omega_n}
\left(\dfrac{1-d}{1+d}\right)^{\frac{n-1}2} \leq\dfrac{(n-1)2^{\frac{n-1}2} V(C) (1-d)^{\frac{n-1}2}}{\omega_n}.\]
\end{lemma}







\section{The double cone} 
\label{doubleconesec}
We will be using the following construction:
Suppose $0 < r_0 < r_1 < r_2,$ and let $C$ be a closed convex subset of $B^n(0, r_2).$ Assume further that $B^n(0, r_0 \subset C),$ and that $C_{r_2} = C \cup \partial B^n(0, r_2) \neq \emptyset.$
Consider now  $\xi\in C_{r_2},$  and the ball $B^n(0, r_1).$  and the cone $H(\xi, r_0),$ which is the convex hull of $B(0, r_0)$ and $\xi.$ The cone $H(\xi, r_0)$ intersects $\partial B^n(0, _r1)$ in a disk $D(\xi, r_1, r_0),$ and we have the following:
\begin{lemma}
\label{doubleconelem}
The disk $D(\xi, r_1, r_0)$ has angular radius $\alpha_{r_0}(r_2, r_1).$
\end{lemma}
\begin{proof}
By rotational symmetry, it suffices to consider the planar case ($n=2$). Let the two tangents to $\partial B^2(0, r_0)$ from $\xi$ be $l_1$ and $l_2,$ and let $l_i\cap \partial B^2(0, r_0)$ be $t_1$ and $t_2,$ respectively. By the Pythagorean theorem, 
$|\xi t_1| = \sqrt{r_2^2-r_0^2}.$ Let now $s_1 = l_1 \cap \partial B^2(0, r_1),$ and let $p$ be the base of the perpendicular dropped from $s_1$ onto $O\xi.$ The triangle $\xi s_1 p$ is similar to the triangle $\xi O t_1,$ and since $|\xi t_1| = \sqrt{r_1^2-r_0^2},$ it follows that 
\[
|p s_1| = r_0 \left(\sqrt{r_2^2-r_0^2}-\sqrt{r_1^2-r_0^2}\right),
\]
and the assertion of the lemma follows immediately.
\end{proof}
which in turn implies that $\Omega_d^K$ contains a (spherical) disk
$D_\xi(\alpha_{r_0}(d))$ of radius $\alpha_{r_0}(d)$ around $\xi.$

\begin{theorem}
\label{doubleconecor1}
With notation as above, let $\Omega_{r}$ be the set of rays from the origin to $C_{r} = C \cap \partial B^n(0, r),$ (identified with a subset of the ``visual sphere'' at the origin -- the unit tangent sphere). Then, the $\alpha_{r_0}(r_2, r_1)$ neighborhood of $\Omega_{r_2}$ is contained in $\Omega_{r_1}.$
\end{theorem}
\begin{proof}
Consider a point $\eta \in \Omega_{r_2}.$ By Lemma \ref{doubleconelem}, the cone $J_{\alpha_{r_0}(r_2, r_1)}(0)$ with the vertex at the origin and angle $\alpha_{r_0}(r_2, r_1)$ is contained in $\Omega_{r_1},$ which is precisely the statement of the Corollary.
\end{proof}
\begin{remark}
The cones $H$ and $J$ give this section its name.
\end{remark}

\section{Applications to limit sets}
\label{limapps}
Let $C$ be a convex body in $\ahn.$ We will say that \emph{the limit
  set} of $C$ -- denoted by $C_\infty$ -- is the intersection of (the closure of) $C$ with
the ideal boundary of $\ahn.$ In the Klein model,
\[
K(C_\infty) = K(C) \cap \partial B^n(0, 1).
\]
Note that in the Klein model we can
identify the ideal boundary of $\ahn$ with the unit tangent sphere at
the origion. With that identification, using the notation of Corollary
\ref{rgenmeas}, we can define 
\[
C_\infty = \bigcap_R \Omega_R(C). = \bigcap_d \Omega_d^K(C).
\]
In the sequel, we will assume that $C$ has non-empty interior, and
from now on, all computations will be in the Klein model. We then assume
particular, there is a ball $B_0$  of radius $r_0$ centered on the origin and
contained in $K(C).$ 

Assume now that $C$ has finite volume $V(C).$ Theorem \ref{doubleconecor1} allows us to strengthen Lemma \ref{Kmeas} as follows:
\begin{lemma}
\label{Kmeasgen}
\begin{eqnarray*}
\mu((\Omega_{d_1}^K(C))_{\alpha_{r_0}(d_1, d_2)})  &  \leq \\ \dfrac{(n-1) 2^{n-1}V(C)}{\omega_n}
\left(\dfrac{1-d_2}{1+d_2}\right)^{\frac{n-1}2} &  \leq\\ \dfrac{(n-1)2^{\frac{n-1}2} V(C) (1-d_2)^{\frac{n-1}2}}{\omega_n}.
\end{eqnarray*}
\end{lemma}
We are now ready to show:
\begin{theorem}
\label{hdimest}
Let $C$ be a convex set in $\ahn$ of finite volume with nonempty interior. Then, the (upper)
Minkowski dimension of $C_\infty$ is at most $(n-1)/2.$
\end{theorem}
\begin{proof}
 Set $d_1=1$ in the statement of Lemma \ref{Kmeasgen}. By Lemma \ref{alphabaseest} we see that 
\[
\Omega_1^K(C)_{\asin (r_0 (1-d))}\subseteq \Omega_d^K(C)_{\asin r_0 \frac{d}{1-d}} \subseteq (\Omega_{d}^K(C))_{\alpha_{1, d}},
\]
and so (setting $\epsilon = \asin (r_0(1-d))$)
\[
\mu\left(\Omega_1^K(C)_\epsilon\right) \leq \dfrac{(n-1)2^{\frac{n-1}{2}}V(C)\sin^{\frac{n-1}2} \epsilon}{r_0^{\frac{n-1}{2}} \omega_n}.
\]
Letting $\epsilon$ tend to $0,$ we see that the measure of $\Omega_1^K(C)_\epsilon$ is bounded above by a constant times  $\epsilon^{\frac{n-1}2},$ whence the result.
\end{proof}
\begin{corollary}
\label{hdimestcor}
With $C$ as above, the Hausdorff dimension of $C_\infty$ is at most $(n-1)/2.$
\end{corollary}
\begin{proof}
The Minkowski dimensions (upper and lower) are both upper bounds on the Hausdorff dimension.
\end{proof}
\subsection{Are the results on dimension sharp?}
As observed by Peter Storm, the bound of Theorem \ref{hdimest} is
clearly \emph{not} sharp in 
dimension $2.$ There, since the area of an ideal triangle is always
$\pi,$ it is easy to see that if $C$ has finite area, $C_\infty$ is a
\emph{finite} set, and hence any reasonable dimension of $C_\infty$
equals $0.$

On the other hand, Eq.~\eqref{milnoreq} indicates that the hyperbolic
volume element at $q$ is proportional to the Euclidean volume element
divided by $r^{(n+1)/2}.$ This indicates that the convex hull $C$ of a
(very) small totally geodesic disk $D_k$ (of dimension $k$) on
$\partial B^n(0, 1)$ and a ball $F$ in the interior of $B(0, 1)$ looks,
near the ideal boundary, as a Cartesian product of $D$ and the cone
from a point $p\in D$ onto a section $F^\perp$ of $F$ orthogonal to
$D.$ Since the dimension of $F^\perp$ equals $n-k,$ we see that we
have shown:
\begin{theorem}
\label{smoothvol}
Let $M$ be a piecewise smooth embedded submanifold of $\partial B^n(0,
1)$ of dimension $d,$ and let $C(M)$ be the convex hull of $M.$ Then
the hyperbolic volume of $C(M)$ is finite  if and only if $k$ is \emph{smaller}| than
$n-(n+1)/2 = (n-1)/2.$ 
\end{theorem}
\begin{remark}
The regularity required in the statement of Theorem \ref{smoothvol} is
not very onerous: $C^1$ is certainly sufficient; presumably
rectifiable is also. 
\end{remark}
Theorem \ref{smoothvol} indicates that the bound of Theorem
\ref{hdimest} is \emph{sharp}  for piecewise smooth sets and when $n$ is even. For arbitrary sets, we show a lower bound (at least when $n$ is $3$) in Section \ref{explower}.
\begin{remark}
By the results of B. Colbois and P. Verovic \cite{colboishilbert}, the results of this note
apply essentially without change to convex bodies in the Hilbert
metric on arbitrary smooth convex domains.
\end{remark}
\section{Applications to central sections}
\label{centralsec}
In this section we will apply the above results to the following question: 
\begin{quotation}
Suppose we have a convex set $C$ with nonempty interior and finite volume $V(C)$  in $X,$ and a point $p \in C.$ For each $k$-dimensional plane $\Pi$ through $p,$ consider $\Pi_C = \Pi \cap C,$ and  let $d(\Pi) = \max_{x\in \Pi_C}(d(x, p)$ -- in other words, $d(\Pi)$ (not necessarily finite) is the radius of the smallest sphere containing $\Pi_C.$ The question, then, is do we have any \emph{upper} bound on the \emph{smallest} $d(\Pi)?$
\end{quotation}
In this form, the question is not hard to answer using our results above.
First, we will need the following standard fact (see, eg, \cite[page 4]{milschechtbook}):
\begin{theorem}
\label{grassfubini}
Let $1\leq k \leq n,$ let $\zeta \in G_{n, k},$ where $G_{n, k}$ is the Grassmannian of $k$ planes through the origin in $\mathbb{R}^n.$ Denote by  $S(\zeta) = \mathbb{S}^{n-1}\cap \zeta$ the unit sphere of $\zeta.$
Then
\[
\int_{\mathbb{S}^{n-1}} f d\nu = \int_{G_{n, k}}\int_{S(\zeta)} f(t) d\nu_\zeta(t)d\nu(\zeta)
\]
for all $f\in L^1(\sphere^{n-1}),$ where $\nu_\zeta$ is the normalized Haar measure on $S(\zeta),$ $\nu$ on the left is the normalized Haar measure on $\sphere^{n-1}$ and on the right on $G_{n, k}.$
\end{theorem}
We will be applying Theorem \ref{grassfubini} to the indicator function $f_\epsilon(M)$ of the $\epsilon$-neighborhood $M_\epsilon$ of a set $M\subseteq \sphere^{n-1}.$ Every $k$-plane which intersects $M$ intersects $M_\epsilon$ in at least an $\epsilon$-ball, and therefore if \emph{every} $k$-plane intersects $K,$ we have the inequality:
\begin{equation}
\label{epsilonineq}
\mu(M_\epsilon)\geq \sphvol{k}(\epsilon),
\end{equation}
which implies:
\begin{theorem}
\label{basicthm}
There is a $k$-plane $\Pi_k$ through the origin such that \[(\Pi_k \cap C) \subseteq B^n\left(0, \frac12\log \frac{1+d}{1-d}\right)\] if 
\[
\dfrac{\mu((\Omega_d^K)_\epsilon)}{\sphvol{{k-1}}(\epsilon) } < 1
\]
for \emph{some} $\epsilon > 0.$ 
\end{theorem}

Now, let $M = \Omega_d^K(C).$ where $C$ satisfies the hypotheses of the beginning of this note (in particular, contains a ball of radius $r_0$ around the origin). By Lemma \ref{Kmeasgen}, for any $d_2 < d,$ we have
\[
\nu((\Omega_d^K)_{\alpha_{r_0}(d, d_2)}) \geq \dfrac{(n-1)2^{\frac{n-1}2}V(C)(1-d_2)^{\frac{n-1}2}}{\omega_n^2}.
\]
We assume in the sequel that $V(C)$ is large, and that $d, d_2$ are close to $1.$ Under those assumptions, by Lemma \ref{alphabaseest} we have
\[
( \Omega_d^K)_{r_0(d_2 - d)} \subset (\Omega_d^K)_{\alpha_{r_0}(d, d_2)} 
\]
Setting $\epsilon = 1 - d$ and $\epsilon_2 = 1 - d_2,$ we also have (for any $k>0$), 
\[
V_k(r_0(\epsilon_2-\epsilon)) \sim \kappa_{k-1} r_0^{k-1} (\epsilon_2 - \epsilon)^{k-1}.
\]
Thus, if $\Omega_d^K$ intersects every plane, by Eq.~\eqref{epsilonineq},
\[
(n-1)2^{\frac{n-1}2}V(C)\epsilon_2^{\frac{n-1}2}/\omega_n^2 \geq \kappa_{k-1} r_0^{k-1}(\epsilon_2 - \epsilon)^{k-1}/\omega_k
\]
for \emph{every} $\epsilon_2>\epsilon.$
Writing $\epsilon_2 = \epsilon(1+x),$
we get 
\[
(n-1)2^{\frac{n-1}2}V(C)(1+x)^{\frac{n-1}2}\epsilon^{\frac{n-1}2}/\omega_n^2 \geq \kappa_{k-1} r_0^{k-1}x^{k-1}\epsilon^{k-1}/\omega_k,
\]
or
\begin{equation}
\label{preest}
\epsilon^{\frac{n-1}2 - (k-1)} \geq\dfrac{\kappa_{k-1}r_0^{k-1}\omega_n^2}{(n-1)2^{\kappa_{k-1}\frac{n-1}2}V(C)}\dfrac{x^{k-1}}{(1+x)^{\frac{n-1}2}},
\end{equation}
for all $x>0.$

Applying Lemma \ref{maxest} to the estimate \eqref{preest}, with $m=\dfrac{n-1}{2},l = {k-1},$ we see that in order for $\Omega_{1-\epsilon}^k$ to intersect every plane through the origin,
we must have
\begin{equation}
\label{postest}
\epsilon\geq \frac{n+1-2k}{2}\left(\dfrac{\kappa_{k-1}r_0^{k-1}\omega_n^2}{\omega_k V(C)}\dfrac{(k-1)^{k-1}}{(n-1)^{(n+1)/2}}\right)^{2/(n+1-2k)}.
\end{equation}
If we assume in addition that $k\ll n,$ the estimate \eqref{postest} simplifies further to:
\begin{equation}
\label{postest2}
\epsilon > \frac{1}{2}\left(\dfrac{r_0^{k-1}\omega_n^2}{V(C)}\right)^{2/(n+1-2k)},
\end{equation}
which simplifies further using Stirling's formula to:
\begin{equation}
\label{postest3}
\epsilon > \frac{2\pi^2 e^2}{ n^2}\left(\dfrac{r_0^{k-1}}{V(C)}\right)^{2/(n+1-2k)}.
\end{equation}
\subsection{Hyperbolic Space}
The corresponding hyperbolic radius is given by 
\[
r = \frac12\log\left(\frac{2-\epsilon}{\epsilon}\right) \sim \frac12 (\log 2 - \log \epsilon) \sim \frac12 \log 2 - \frac12 \log \epsilon,
\]
so we have
\begin{theorem}
\label{hypradversion}
Let $C \in \ahn$ be a convex set of \emph{large} volume $V(C)$ which contains a ball radius $r_0 \ll 1$ around a point $p \in C.$ Then if $k \leq (n-1)/2$ there exists a $k$-dimensional plane $\Pi_k$ through $p,$ such that $C\cap \Pi_k$ is contained in $B(p, r),$ as long as
\begin{multline}
r= \log 2 - \frac12 \log (n+1-2k)  - \frac{1}{n+1-2k} \times \\ \left((k-1)\log (k-1) r_0  + \log(\kappa_{k-1}\omega_n^2/\omega_k)   -  \log V(C) - \frac{n+1}2 \log (n-1)\right).
\end{multline}
For $n \gg k,$ there is the  asymptotic version:
\[
r =  \frac{\log V(C) - (k-1) \log r_0}{n+1-2k} + \log n - \log \pi - 1.
\]
\end{theorem}
\begin{remark}
A comparison of Theorem \ref{hypradversion} and the estimate \eqref{stirlingball} indicates that we ``lose'' roughly $\frac12 \log n$ for the diameter of sections of arbitrary convex bodies of volume $V$ versus a ball of the same volume.
\end{remark}
\begin{example}
A non-asymptotic example is when $n=3, k=1.$ Then we get the estimate
\begin{multline}
r = \log 2 - \frac12 \log2 -\frac12\left(2 \log 4\pi - \log 2 - \log V(C) - 2\log 2\right) =\\
\frac12\log V(C)-\log2\pi,
\end{multline}
valid for large $V(C).$
\end{example}

\section{Convex sets in $\mathbb{E}^n.$}
\label{euclid}
Here, we use the techniques developed above to analyze what we can show about convex sets in $\mathbb{E}^n.$ Related work can be found in \cite{klartagmest} and references therein; Klartag's results are asymptotically sharper, but since our methods seem completely different and more geometric, and the estimates we obtain are quite concrete, the current section seems to be of interest.  Let $C$ be such a convex set, and, as before, we assume that $C$ has positive volume (hence nonempty interior). For simplicity, set $p=\mathbb{0}.$ Assume that $B(\mathbb{0}, r_0) \subset C.$ It is clear that the diameter of $C$ is bounded, since the volume of a right cone in $\mathbb{E}^n$ grows linearly with the altitude, so the questions about the dimension of $C_\infty$ do not come up. However, the questions of diameter of planar sections as in Section \ref{centralsec} are interesting (especially as they are connected to the extensive work on the Busemann-Petty problem, as in \cite{koldobskybp,zhangpos,gardnerbp,grin}, and it is not difficult to extend our methods to this setting.
\subsection{Volume estimates}
\label{eucvolest}
By  the standard formulae for Euclidean spheres and balls in Eq. \eqref{eucvols} together with Corollary \ref{rgenmeas} we get the following estimate on the 
 the visual measure of the set $\Omega_r(C)$  of directions where the ray of radius $r$ from the origin is contained in $C:$
\begin{equation}
\label{eucmeasure}
\mu(\Omega_r(C)) \leq \dfrac{V(C)}{\kappa_n r^n},
\end{equation}
where, as before, $V(C)$ denotes the volume of $C.$
\subsection{Double cone lemma estimates}
The proof and statement of the Double Cone Lemma \ref{doubleconelem} go through without change; we state the result for convenience here:
\begin{lemma}
\label{doublecone2}
Let $r_1 > r_2 >r_0 0,$ and let $\alpha_{r_0}(r_1, r_2)$ be as in Definition \ref{alphadef}.
Then the $\alpha_{r_0}(r_1, r_2)$ neighborhood on $\Omega_{r_1}(C)$ is contained in $\Omega_{r_2}(C).$
\end{lemma}
\subsection{Applications to finding round sections}
As before, our basic tool is:
\begin{theorem}
\label{basicthmeuc}
There is a $k$-plane $\Pi_k$ through the origin such that \[(\Pi_k \cap C) \subseteq B(0, r)\] if 
\[
\dfrac{\mu((\Omega_r)_\epsilon)}{V^{k-1}(\epsilon) } < 1
\]
for \emph{some} $\epsilon > 0.$ Above, $V^{k-1}(\epsilon)$ denotes the normalized volume of the spherical ball of radius $\epsilon.$
\end{theorem}
Using Eq. \eqref{eucmeasure}, Lemma \ref{doublecone2}, and Theorem \ref{basicthmeuc}, we get:
\begin{corollary}
\label{basiccoreuc}
There is a $k$-plane $\Pi_k$ through the origin, such that \[(\Pi_k \cap C) \subseteq B(0, r),\]
if 
\[
\dfrac{V(C)}{\omega_n \kappa_n r_1^n V^{k-1}(\alpha_{r_0}(r, r_1))} < 1
\]
for some $0< r_1 < r.$
\end{corollary}
Combining Lemmas \ref{sphvollem},\ref{alphaaux},\ref{alphabaseest}, and \ref{maxest} we see:
\begin{lemma}
Assuming, as before, that $r_1 = r/(1+x),$ for some $x > 0,$
\label{corest}
\begin{multline}
\min_{x >0} \dfrac{V(C)}{\omega_n \kappa_n r_1^n V^{k-1}(\alpha_{r_0}(r, r_1))} < \min_{x>0}\dfrac{V(C)(1+x)^n\omega_{k}}{\kappa_{k-1}\omega_n \kappa_n r^n r_0^{k-1}x^{k-1}} \\=\dfrac{V(C) \omega_k}{\kappa_{k-1}\omega_n \kappa_n r^n r_0^{k-1}} \dfrac{n^n}{(k-1)^{k-1}(n-k+1)^{n-k+1}}.
\end{multline}
\end{lemma}
And so finally, using Corollary \ref{basiccoreuc}, we get:
\begin{theorem}
\label{mainthmeuc}
For any convex set $C\subset \mathbb{E}^n$ of volume $V(C)$ and containing a ball of radius $r_0$ centered at the origin, and $k\leq n$ there is a plane $\Pi_k$ through the origin such that  $\Pi_k \cap C\subseteq B^n(0, r),$ for
\[
r = n\left(\dfrac{V(C) \omega_k}{\kappa_{k-1}\omega_n \kappa_n r_0^{k-1} (k-1)^{k-1} (n-k+1)^{n-k+1}}\right)^{\frac1n}.
\]
\end{theorem}
\begin{corollary}
\label{asympcoreuc}
If $k \ll n,$ we can simplify the estimate of Theorem \ref{mainthmeuc} to 
\[
r \sim \frac{n}{2\pi e}\dfrac{V(C)^{1/n}}{r_0^{(k-1)/n}}..
\]
\end{corollary}
Note that the radius of the ball in $\mathbb{E}^n$ of volume $V(C)$ is (for large $n$) approximately
\[\sqrt{\frac{n}{2\pi e}} V(C)^{1/n},\] so we lose a factor of $\sqrt{n/2\pi e}.$

\section{Useful estimates}
\label{theests}
To continue, we will need the following:
\begin{lemma}
\label{maxest}
For any $x \geq 0,$ and any $m>l>0,$ 
\[
g(x)=x^l/(1+x)^m \leq \dfrac{l^l (m-l)^{m-l}}{m^m}.
\]
\end{lemma}
\begin{proof}
Since $g(0) = g(\infty) = 0,$ the (smooth) function $g(x)$ achieves its maximum at some $x_0$ in $(0, \infty).$
Since $g(x)$ is positive on the positive real axis, its natural logarithm $h(x)$ is everywhere defined, and achieves its maximum at $x_0$ also, 
since 
\[
h^\prime(x) = l/ x - m /(1+x),
\]
we must have 
\[
0 = h^\prime(x_0) = \dfrac{l}{x_0} - \dfrac{m}{1+x_0} = \dfrac{l - (m-l)x_0}{x_0 (1+x_0)},
\]
and so 
\[
x_0 = \dfrac{l}{m-l}.
\]
and
\[
g(x_0) = \dfrac{\left(\frac{l}{m-l}\right)^l}{\left(1+\frac{l}{m-l}\right)^m}.
\]
Since $1+l/(m-l) = m/(m-l),$ the result follows.
\end{proof}

To get concrete estimates, let us write $r_1 = r/(1+x),$  and observe:
\begin{lemma}
\label{alphabaseest}
For $x>r_0,$ $\alpha_{r_0}(r, r/(1+x)) > \asin( r_0 x).$
\end{lemma}
\begin{proof}
It is enough to show that for $x \in (r_0, 1)$
\begin{equation}
\label{alphaarg}
\dfrac{\sqrt{r^2-r_0^2}-\sqrt{\frac{r^2}{1+x}^2 - r_0^2}}{\frac{r}{1+x}} > x,
\end{equation}
by monotonicity of $\asin.$ 
The left hand side of Eq.~\ref{alphaarg} can be rewritten as
\[
(1+x)\left(\sqrt{1-\frac{r_0^2}{r^2}} - \sqrt{\frac{1}{(1+x)^2}- \frac{r_0^2}{r^2}}\right).
\]
By Lemma \ref{alphaaux}, the expression inside the parentheses is smaller than $x/(1+x),$ and so the assertion of the Lemma follows.
\end{proof}
\begin{lemma}
\label{alphaaux}
For $a \in (0, 1),$ $x \in (a, 1),$
\[
\sqrt{1-a^2} - \sqrt{x^2-a^2} > 1 - x,
\]
with equality if and only if $x=1.$
\end{lemma}
\begin{proof}
For $x=1$ the two sides of the inequality are equal to $0.$ Otherwise,
\[
\dfrac{d(\sqrt{1-a^2}-\sqrt{x^2-a^2})}{dx} = -\dfrac{x}{\sqrt{x^2-a^2}} = \dfrac{-1}{\sqrt{1-a^2/x^2}} < -1 = \dfrac{d(1-x)}{dx},
\]
whence the result follows.
\end{proof}
\begin{remark}
The proof above actually shows that
\[
\sqrt{1-a^2}-1-a > \sqrt{1-a^2} - \sqrt{x^2-a^2} - 1 - x > 0
\]
for the intervals in question (since the derivative of the middle expression  is strictly negative, and the left and right expressions are the values at the two endpoints of the interval $(a, 1).$)
\end{remark}
\begin{lemma}
\label{sphvollem}
Let $V^l(r)$ be the normalized volume of the spherical ball of radius $r.$ Then,
\[
V^l(r) > \frac{\kappa_l}{\omega_{l+1}}\sin^n r
\]
\end{lemma}
\begin{proof}
We know that 
\[V^l(r) = \frac{\omega_k}{\omega_{k+1}}\int_0^r \sin^{l-1}\eta d\eta.\]
Making the substitution $\eta = \asin \rho,$ we see that
\[V^{k}(r) = \frac{\omega_k}{\omega_{k+1}}\int_0^{\sin r} \frac{\rho^{n-1}}{\sqrt{1-\rho^2}}d\rho > \frac{\omega_k}{k \omega_{k+1}} \sin^k r = \dfrac{\kappa_k}{\omega_{k+1}}\sin^k r.
\]
\end{proof}

\section{Explicit lower bound on limitset dimension in $\mathbb{H}^3.$}
\label{explower}
In this section we prove
\begin{theorem}
\label{hausdorff3}
For any $\beta < 1,$ there is a set $S_\beta$ in $\partial \mathbb{H}^3,$ such that Hausdorff dimension of $S_\beta$ equals $\beta,$ while the volume of the convex hull of $S_\beta$ is finite.
\end{theorem}
Theorem \ref{hausdorff3} shows that our dimension estimate is sharp, although it does leave open:
\begin{question}
\label{dim1}
Does there exist a set $S_1 \subset \partial \mathbb{H}^3$ with Hausdorff dimension of $S_1$ \emph{equal} to $1$ and such that the convex hull of $S_1$ has finite volume?
\end{question}

The proof of Theorem \ref{hausdorff3} is by explicit construction, and is contained in section \ref{cantorvol}. The needed facts concerning the volume of ideal simplices in $\mathbb{H}^3$ are contained in section \ref{simpsec}.
\subsection{Ideal simplices in $\mathbb{H}^3$}
\label{simpsec}
Consider $B^3(0, 1),$ viewed as the Klein model of $\mathbb{H}^3,$ and consider points \[(0, 0, 1), (0, 0, -1), (1, 0, 0), (\cos \theta, \sin \theta, 0)\] on the sphere $\partial B^3(0, 1).$ The convex hull of these four points is an \emph{ideal tetrahedron} $T_\theta.$ Under stereographic projection (from the north pole $(0, 0, 1)$ the four points go to $\infty, 0, 2, 2 \exp i \theta,$ and so the dihedral angles of $T_\theta$ are $\theta, \pi/2-\theta/2, \pi/2 - \theta/2.$ It follows that the hyperbolic volume of $T_\theta$ is given by 
\begin{equation}
\label{volformula}
V(T_\theta) = \lobfn(\theta) + 2 \lobfn(\pi/2-\theta/2),
\end{equation}
where $\lobfn(x)$ denotes the Lobachevsky function:
\begin{equation}
\lobfn(x) = - \int_0^x \log 2 |\sin t| d t.
\end{equation}
Many properties and applications of the volumes of ideal simplices are discussed in \cite{ann2}, but here we will only need the following simple result:
\begin{lemma}
\label{vol3}
When $\theta \ll 1,$ we have
\[
V(T_\theta) \sim -\theta \log\theta.
\]
\end{lemma}
\begin{proof}
It is clear (for geometric reasons) that $\lim_{\theta\rightarrow 0} V(T_\theta) = 0.$ Using Eq. \eqref{volformula}, we see that this implies that $\lobfn(\pi/2) = 0,$ and the statement of the Lemma then follows from the fundamental theorem of calculus.
\end{proof}
\subsection{Volume of the convex hull of Cantor sets}
\label{cantorvol}
We will be using the ``standard'' family of Cantor sets $C(\alpha),$ where $0<\alpha < 1/2.$ Such a set is obtained by starting with the interval $[0, 1],$ then deleting the interval $(\alpha, 1-\alpha),$ then applying the construction to each of the remaining intervals, and so on recursively. The usual middle thirds Cantor set is $C(1/3).$ It is well known that the Hausdorff dimension of $C(\alpha)$  equals 
\[
\dfrac{\log 2}{\log \frac{1}{\alpha}},
\]
see \cite{mattilabook} for proof and discussion.
Consider the set  $S_\alpha \subset \partial B^3(0, 1)$ consisting of the North pole -- $(0, 0, 1)$, the South pole -- $(0, 0, -1),$ and a Cantor set on the equator. An example of such a Cantor set can be obtained by identifying the set $(x, y, 0), y \geq 0 \subset \partial B^3(0,1)$ with the unit interval, and then constructing a Cantor set $C(\alpha)$ in that interval.. We claim that the convex hull of $S$ has finite volume. Indeed, the convex hull of S is a union of ideal tetrahedra $T_\theta,$ as above. Each tetrahedron corresponds to an interstitial region in the Cantor construction, so that, for example, the first stage contributes a $T_{\pi(1-2\alpha)},$ the second stage contributes two copies of $T_{\pi\alpha(1-2\alpha)},$ and the $n$-th stage contributes $2^{n-1}$ copies of $T_{\pi/(1-2\alpha)\alpha^{n-1}}.$ It follows that the volume of the convex hull of $S$ is:
\begin{equation}
\label{volsum}
  \sum_{n=1}^\infty 2^{n-1}  V(T_{\pi(1-2\alpha)\alpha^{n-1}}).
\end{equation}
For sufficiently large $n,$ Lemma \ref{vol3} tells us that $V(T_{\pi/(1-2\alpha)\alpha^{n-1}})$ is of the order of $-(n-1) \alpha \log \alpha,$ and so the sum in Eq.~\eqref{volsum} converges, thus the volume is finite. Note, however, that the volume of the convex hull of $S_\alpha$ goes to infinity as $\alpha$ tends to $1/2,$ and thus the Hausdorff dimension of $C(\alpha)$ tends to $1.$ 
\section{Higher Dimension}
\label{cantorhigher}
The lower bounds for $n=3$ seem to depend on an explicit formula for the volume of ideal simplices and  the geometry of one-dimensional Cantor sets. It turns out that both aspects can be generalized to higher dimensions. The sets we will use will be \emph{generalized Sierpinski carpets}, $K(M),$ constructed as follows:

We start with the unit cube $K_0=K^n = [0, 1]^n \subset \reals^n.$ At the next step we subdivide $K_0$ into $N^n$ equally sized cubes, each of side-length $1/N.$ Number these cubes from $1$ to $N^n.$  If $M \subseteq {1, \dotsc, N}^n,$ delete all the cubes whose indices are not in $M,$ to obtain the set $K_1(M).$ Now, apply the process to each of the $M$ remaining cubes to obtain $K_2(M),$ and so on. The final carpet is the limiting object:
\[
K(M) = \bigcap_{k=0}^\infty K_k(M).
\]
The standard Sierpinski carpet is obtained by setting $n=2,$ $N=3,$ $M=\{(1, 1), (1, 2), (1, 3),(2, 1), (2, 3), (3, 1), (3, 2), (3, 3) \},$ where the numbering goes from top right to bottom left.  The unit interval can be obtained in this setting by letting $M=\{(2, 1), (2, 2), (2, 3)\}.$ The middle thirds Cantor set is obtained by setting $n=1,$ $N=3,$ $M=\{1, 3\},$ the numbering going from left to right.

The following Theorem follows immediately from \cite[Section 4.12]{mattilabook}:
\begin{theorem}
\label{sierpdim}
The Hausdorff dimension of $K(M)$ equals $\log |M|/\log N.$
\end{theorem}
Theorem \ref{sierpdim} gives us the well-known values $\log2/\log 3,=0.63,$ $ \log 8/\log 3=1.89, $ $1$ for the Hausdorff dimensions of the Sierpinski carpet, the middle thirds Cantor set, and the unit interval, respectively.

Theorem \ref{sierpdim} has the following obvious corollary:
\begin{corollary}
\label{approxdim}
In $\reals^n$ there are Sierpinski carpets of Hausdorff dimension arbitrarily close (but not equal) to $n.$
\end{corollary}
\begin{proof}
Let $N=2 L,$ and let $M$ be the set of $n$-tuples $(i_1, \dotsc, i_n)$ where $i_k = k \mod 2.$ The cardinality of $M$ equals $L^n,$ and so the Hausdorff dimension of $K(M)$ equals $n-n \log 2/\log N.$ For $N \gg n, $ this will be close to $n.$
\end{proof}
The set in the example above is constructed in such a way that the sets at the $k$-iteration of the construction have diameter exponentially decreasing with $k.$ The same construction as in three dimensions (\emph{mutatis mutandis}) gives us the result that in any \emph{odd} dimension, our bound on the Hausdorff dimension is sharp, in other words:
\begin{theorem}
\label{hausdorffn}
For any $\beta < 1,$ there is a set $S_\beta$ in $\partial \mathbb{H}^3,$ such that Hausdorff dimension of $S_\beta$ equals $\beta,$ while the volume of the convex hull of $S_\beta$ is finite.
\end{theorem}

\bibliographystyle{plain}
\bibliography{rivin}
\end{document}